\documentclass[12pt,twoside,a4paper]{amsart}
\usepackage{stmaryrd}
  
\usepackage[T1]{fontenc}
\usepackage[latin1]{inputenc}
\usepackage[left=30mm,top=30mm,bottom=30mm]{geometry}

\usepackage{amssymb}
\usepackage{mathtools}
\usepackage{lmodern}

\usepackage[square,numbers]{natbib}
\usepackage{enumitem}

      \theoremstyle{plain}
      \newtheorem{Thm}{Theorem}[section]
      \newtheorem{Lem}[Thm]{Lemma}

      \theoremstyle{definition}
      
      \theoremstyle{remark}
      \newtheorem{Remark}[Thm]{Remark}
\newcommand{\N}{{\mathbb N}}     %
\newcommand{\R}{{\mathbb R}}     %
\newcommand{\bcp}{\mathrm{BC}_p}
\newcommand{\B}{\mathrm{B}}
\newcommand{\Ber}{\mathrm{B}}
\renewcommand{\epsilon}{\varepsilon}\renewcommand{\phi}{\varphi} %
\renewcommand{\rho}{\varrho}\renewcommand{\theta}{\vartheta}     %
\newcommand{\Kappa}{\mathrm{K}} \newcommand{\Mu}{\mathrm{M}}     %
\newcommand{\K}{\mathrm{K}}                                      %
\newcommand{\defn}[1]{\emph{#1}}   
\renewcommand{\[}{\begin{eqnarray*}}\renewcommand{\]}{\end{eqnarray*}}
\newcommand{\la}{\begin{eqnarray}}\newcommand{\al}{\end{eqnarray}}
\newcommand{\id}{{\mbox{\rm id}}}

\DeclareMathOperator*{\bigconv}{\mbox{\LARGE$\ast$}}
\newcommand{\pqu}{\overline{p}}
\newcommand{\1}{\mathbf 1}       

\begin{document} 
\author[Mattner]{Lutz Mattner} 
\address{Universit\"at Trier, Fachbereich IV -- Mathematik, 
	54286~Trier, Germany}
\email{mattner@uni-trier.de}
\author[Tasto]{Christoph Tasto} 
\address{Universit\"at Trier, Fachbereich IV -- Mathematik, 
	54286~Trier, Germany}
\email{tasto@uni-trier.de}

\title
[Confidence intervals for average success probabilities]%
{Confidence intervals for\\ average success probabilities}

\begin{abstract}
We provide Buehler-optimal one-sided and some valid two-sided 
confidence intervals for the average success probability of a possibly inhomogeneous
fixed length Bernoulli chain, based on the number of observed successes.
Contrary to some claims in the literature, the one-sided 
Clopper-Pearson intervals for the homogeneous case are not completely robust
here, not even if applied to hypergeometric estimation problems.
\end{abstract}



\thanks{This research was partially supported by DFG grant MA 1386/3-1.}

\maketitle

\section{Introduction and results}\label{Sec:1}
The purpose of this paper is to provide optimal one-sided (Theorem~\ref{Thm:lower})
and some valid two-sided (Theorems~\ref{Thm:main} and~\ref{Thm:Two-sided-CP_valid})
confidence intervals for the average success probability of a possibly inhomogeneous
fixed length Bernoulli chain, based on the number of observed successes.
For this situation, intervals proposed in the literature known to us are, 
if at all clearly specified, 
in the one-sided case either not optimal or erroneously claimed to be valid,  see 
Remarks \ref{Remark:Agnew-interval_one-sided} and~\ref{Rem:Wrong_claims}  below,
and in the two-sided case either improved here, see Remark~\ref{Rem:Two-sided_CP_Vs_Agnew}, 
or not previously proven to be valid. 

To be more precise, let $\B_p$ for $p\in[0,1]$, $\B_{n,p}$ for $n\in\N_0$ and $p\in[0,1]$, 
and $\bcp\coloneqq \bigconv\nolimits_{j=1}^n \Ber_{p_j}$ for $n\in\N_0$ and $p\in[0,1]^n$
denote the Bernoulli, binomial, and Bernoulli convolution (or Poisson-binomial)
laws with the indicated parameters. 
For $a,b\in \R\cup\{-\infty,\infty\}$ let 
$\mathopen]a,b\mathclose]\coloneqq \{x\colon a < x \leq b\}$
and let the other intervals be defined analogously.
Then, for $n\in\N$ and $\beta\in\mathopen]0,1\mathclose[$,
and writing $\pqu\coloneqq\frac1n\sum\nolimits_{j=1}^n p_j$ for $p\in [0,1]^n$,
we are interested in $\beta$-confidence regions for   
the estimation 
problem
\begin{align}                                                   \label{stmdl}
 \left( \left(\bcp\colon p\in[0,1]^n\right), [0,1]^n\ni p\mapsto\pqu\right),
\end{align}
that is, in functions $\Kappa\colon\{0,\ldots,n\}\rightarrow 2^{[0,1]}$ satisfying 
$ \bcp\left(\K \ni \pqu    \right) \ge  \beta$ for $p\in[0,1]^n$. Clearly,
every such $\Kappa$ is also a $\beta$-confidence region for the binomial
estimation problem
\begin{align}                             \label{binom2}
\left( (\B_{n,p}\colon p\in[0,1]),\id_{[0,1]} \right),
\end{align}
that is, 
satisfies 
$\B_{n,p}\left(\K \ni p\right) \ge \beta$ for $p\in[0,1]$,
but the converse is false by Remark~\ref{Rem:KCP_not_robust} below. 
However, a classical Chebyshev-Hoeffding result easily yields the following basic fact.

\begin{Thm}                            \label{Thm:main}
Let $n\in\N$ and $\beta\in\mathopen]0,1\mathclose[$. 
For 
$m\in\{0,\ldots,n\},$ let $\K'_m$ be a $\beta$-confidence
region for 
$\left( (\B_{m,p}\colon p\in[0,1]),\id_{[0,1]}\right)$.
Then a $\beta$-confidence region $\Kappa$ for~\eqref{stmdl} is given by
\begin{align*}
 \Kappa(x)  
   \coloneqq \mspace{-10mu}
   \bigcup\limits_{\genfrac{}{}{0pt}{}{l\in\{0,\dots,x\},}{m\in\{x-l,\dots,n-l\}}} \mspace{-10mu}	
      \left(\tfrac mn \K'_m(x-l) + \tfrac ln \right)
   \,\ \supseteq \,\ \Kappa_n'(x) 
  \quad\text{ for }x\in\{0,\ldots,n\}.
\end{align*}
\end{Thm}
Proofs of the three theorems of this paper are presented in section~\ref{Sec:2} below. 

If 
the above $\K_m'$ are taken  to be one-sided 
intervals of Clopper and Pearson~\cite{Clopper}, then the resulting $\K$ turns out to be Buehler-optimal 
and, if $\beta$ is not unusually small, the formula for $\K$ 
simplifies drastically, as stated in Theorem~\ref{Thm:lower} below  
for uprays:

A set $J\subseteq [0,1]$ is an \defn{upray in} $[0,1]$ 
if $x\in J, y\in[0,1], x\leq y$  jointly imply $y\in J$.
This is equivalent to $J$ being of the form $[a,1]$ 
or $]a,1]$ for some $a\in[0,1]$. A function $\Kappa\colon\{0,\ldots,n\}\rightarrow
2^{[0,1]}$ is 
an \defn{upray} if each of its values $\K(x)$ is an upray in $[0,1]$.

For $\beta\in\mathopen]0,1\mathclose[$ and with
\[
  g^{}_{n}(x) 
  &\coloneqq& g^{}_{n,\beta}(x)\,\  \coloneqq \,\  \text{ the } p \in[0,1] 
  \text{ with } \B_{n,p}(\{x,\dots,n\}) = 1-\beta
\]
for $n\in\N$ and $x\in\{1,\dots,n\}$,
which is well-defined due to the strict isotonicity of $p\mapsto\B_{n,p}(\{x,\dots,n\})$
and which yields in particular the special values
\begin{eqnarray}          \label{Eq:gn(1)}
 g^{}_{n}(1) = 1-\beta^{1/n} &\text{ and }& 
 g^{}_{n}(n) = \left(1-\beta\right)^{1/n}
\end{eqnarray}
and the fact that 
\begin{eqnarray}                                             \label{Eq:g_x_beta_monotone}
 g^{}_{n,\beta}(x) \,\  \text{ is strictly }\left\{\begin{array}{l}\text{increasing}\\ 
    \text{decreasing}\end{array}  \right\} \text{ in } \left\{\begin{array}{l} x\\ \beta \end{array}  \right\},
\end{eqnarray}
the Clopper-Pearson $\beta$-confidence uprays 
$\K^{}_{\mathrm{CP,}n}\colon\{0,\ldots,n\}\rightarrow2^{[0,1]}$
are given by 
\la             \label{clop}
\qquad
\K^{}_{\mathrm{CP,}n}(x) &\coloneqq& \K^{}_{\mathrm{CP,}n,\beta}(x) \,\ \coloneqq \,\  	
   \left\{ \begin{array}{ll}
   \left[0,1\right]               & \text{if } x=0,\\
   \left]g^{}_{n}(x), 1\right]   & \text{if } x\in\{1,\dots,n\}
  \end{array} \right\}
\al
for $n\in\N_0,$ and in particular
\begin{eqnarray*}                                     \label{Eq:KCP(1)}
 \qquad
 \K^{}_{\mathrm{CP,}n}(1)\,\ = \,\ \left]1-\beta^{1/n},1\right]  & \text{ and } &  
 \K^{}_{\mathrm{CP,}n}(n) \,\ =\,\ \left] (1-\beta)^{1/n},1\right]
\end{eqnarray*} 
for $n\in\N.$

An upray $\Kappa\colon\{0,\ldots,n\}\rightarrow 2^{[0,1]}$ is
\defn{isotone} 
if it  is isotone with respect to the usual order on $\{0, \ldots,n\}$ and the order reverse 
to set inclusion on $2^{[0,1]}$, that is, if we have the implication 
\[
 x,y\in \{0, \ldots,n\},\, x < y &\Rightarrow& \Kappa(x) \supseteq \Kappa(y),
\]
and \defn{strictly isotone} if ``$\supseteq$'' above can be sharpened to ``$\supsetneq$''. 
For example, each of the above $\K^{}_{\mathrm{CP,}n}$ is strictly isotone by \eqref{Eq:g_x_beta_monotone} and \eqref{clop}.
An isotone $\beta$-confidence upray 
for~\eqref{stmdl}
is \mbox{\defn{(Buehler-)optimal}} 
(see Buehler~\cite{Buehler} and the recent discussion by Lloyd and Kabaila~\cite{LloydKabaila2010}, prompted by
rediscoveries by Wang~\cite{Wang}) 
if every other 
isotone $\beta$-confidence upray $\Kappa^\ast$ for~\eqref{stmdl}
satisfies $\Kappa(x)\subseteq\Kappa^\ast(x) $ for every $x\in\{0,\ldots,n\}$. 
Finally, a not necessarily isotone $\beta$-confidence upray $\K$ for  
\eqref{stmdl} is \defn{admissible} in the set of all 
confidence uprays for \eqref{stmdl} if for every other $\beta$-confidence upray $\K^\ast$ 
for \eqref{stmdl} with $\K^\ast(x)\subseteq\K(x)$ for each $x\in\{0,\dots,n\}$ we have 
$\K^\ast=\K.$

Let us put 
\begin{eqnarray*}
 \beta_n &\coloneqq& \mathrm{B}_{n,\frac1n}(\{0,1\}) \qquad\text{ for }n\in\N,
\end{eqnarray*}
so that $\beta_1=1$, $\beta_2=\frac34$, $\beta_3 = \frac{20}{27}$, and 
$\beta_n\downarrow \frac2{\mathrm{e}} = 0.735\ldots$, with the strict 
antitonicity of $(\beta_n)$ following from 
Jogdeo and Samuels~\cite[Theorem 2.1 with $m_n\coloneqq n, p_n\coloneqq \tfrac{1}{n}, r\coloneqq0$]{Jogdeo}
so that we have in particular 
\[
  \beta_n &\le& \tfrac34 \qquad\text{ for }n\ge 2.
\]
 
\begin{Thm}                                                         \label{Thm:lower}
Let $n\in\N$ and $\beta \in\mathopen]0,1\mathclose[$, 
and let $\K$ be as in 
Theorem~\ref{Thm:main} with the $\K_m'\coloneqq\K^{}_{\mathrm{CP},m}$ as defined in~\eqref{clop}.
Then $\K$ is the optimal isotone $\beta$-confidence upray for~\eqref{stmdl},
is admissible in the set of all $\beta$-confidence uprays for \eqref{stmdl},  
is strictly isotone, and has the effective level   
$\inf_{p\in[0,1]^n} \bcp \left(\K \ni \pqu\right) = \beta$.
We have
\la                              \label{Eq:K_simple}
 \K(x)&=&\left\{\begin{array}{ll}
                \left[0,1\right]                    & \text{if } x=0,\\
	        \left]\frac{1-\beta}{n},1\right]    & \text{if } x=1,\\
                \left]g^{}_{n}(x), 1\right] & \text{if } x\in\{2,\dots,n\}
                        \text{ and }  \beta \geq \beta_n  .
          \end{array}\right.
\al
\end{Thm}
\begin{Remark} \label{Rem:nested}
Nestedness is preserved by the construction in Theorem~\ref{Thm:main}:
Suppose that we apply   Theorem~\ref{Thm:main} to several $\beta\in\mathopen]0,1\mathclose[$
and that we accordingly write $\K'_{m,\beta}$ and $\K^{}_\beta$ in place of $\K'_m$ and $\K$.
If now $\beta,\tilde\beta\in  \mathopen]0,1\mathclose[$ with $\beta<\tilde\beta$ are such 
that  $\K'_{m,\beta}(x)\subseteq \K'_{m,\tilde\beta}(x) $ holds for $m\in\{0,\ldots,n\}$ and 
$x\in\{0,\ldots,m\}$,
then, obviously,  $\K_{\beta}(x)\subseteq \K_{\tilde\beta}(x)$ holds for $x\in\{0,\ldots,n\}$.
By the second line in~\eqref{Eq:g_x_beta_monotone} and by~\eqref{clop}, 
the Clopper-Pearson uprays are nested, 
and hence so are the uprays of Theorem~\ref{Thm:lower}.
Analogous remarks apply to the confidence downrays of Remark~\ref{Rem:Upper_as_lower}  
and to the two-sided confidence intervals of Theorem~\ref{Thm:Two-sided-CP_valid}.
\end{Remark}
\begin{Remark}                                                    \label{Rem:KCP_not_robust}  
Let $n\ge 2$ and $\beta\in\mathopen]0,1\mathclose[$. As  
noted by Agnew~\cite{Agnew} but ignored by later authors,
compare Remark~\ref{Rem:Wrong_claims} below,
$\K^{}_{\mathrm{CP,}n}$ is not a $\beta$-confidence region
for~\eqref{stmdl}. This is obvious from Theorem~\ref{Thm:lower}
and $\K^{}_{\mathrm{CP,}n}(1) \subsetneq\Kappa(1)$,
using either the optimality of $\Kappa$ and the isotonicity of $\K^{}_{\mathrm{CP,}n}$,
or the admissibility of  $\Kappa$ and $\K^{}_{\mathrm{CP,}n}(x)\subseteq\Kappa(x)$
for every $x$.
If $\beta\ge \beta_n$, then Theorem~\ref{Thm:lower} further implies that the effective level of
$\K^{}_{\mathrm{CP,}n}$ 
as a confidence region for~\eqref{stmdl} is 
\[
 \gamma_n &\coloneqq& 1-n\left(1-\beta^{1/n}\right) \,\ \in \,\ \mathopen]1+\log(\beta),\beta\mathclose[,
\]
as for $p\in[0,1]^n$ with $\pqu \notin \mathopen]\frac{1-\beta}{n},g^{}_{n}(1)\mathclose]$, 
formula \eqref{Eq:K_simple} yields
$\bcp \left(\K^{}_{\mathrm{CP,}n} \ni \pqu\right)$ $=$ $\bcp \left(\K \ni \pqu\right)$ $\geq$ $\beta,$
and considering
$p_1 = n g^{}_{n}(1)\le 1$ and $p_2 = \ldots = p_{n} = 0$ at the second step below yields
\[
 \inf\limits_{\pqu \in \mathopen]\frac{1-\beta}{n},g^{}_{n}(1)\mathclose]} \bcp \left(\K^{}_{\mathrm{CP},n} \ni \pqu \right) 
   &=& \inf\limits_{\pqu \in \mathopen]\frac{1-\beta}{n},g^{}_{n}(1)\mathclose]} \prod\limits_{j=1}^n(1-p_j)\\ 
   &=&  1-n g^{}_{n}(1) \,\ =\,\  \gamma_n.
\]
Since $\gamma_n\downarrow 1+\log(\beta)<\beta$ for $n\rightarrow\infty$,
it follows for $\beta>\frac2{\mathrm{e}}$ that the $\K^{}_{\mathrm{CP,}n}$ 
are not even asymptotic $\beta$-confidence regions for~\eqref{stmdl}.
\end{Remark}
\begin{Remark}                          \label{Remark:Agnew-interval_one-sided}
The only previous $\beta$-confidence upray for~\eqref{stmdl} known to us
was provided by Agnew~\cite[section 3]{Agnew} as 
$\Kappa_{\mathrm{A}}(x)\coloneqq[g^{}_{\mathrm{A}}(x),1]$ with $g^{}_{\mathrm{A}}(0)\coloneqq 0$ 
and $g^{}_{\mathrm{A}}(x)\coloneqq g^{}_n(x) \wedge \frac{x-1}n$ for $x\in\{1,\ldots,n\}$.
But $\Kappa_{\mathrm{A}}$ is strictly worse than the optimal isotone $\Kappa$ 
from  Theorem~\ref{Thm:lower}, since $\Kappa_{\mathrm{A}}$ is isotone as well, 
with $\Kappa_{\mathrm{A}}(1)=[0,1] \supsetneq \Kappa(1)$.
On the other hand,  Lemma~\ref{Lem:Agnew_bounds_simplified} below shows 
that actually $g^{}_{\mathrm{A}}(x)=g^{}_n(x)$ for  $\beta\ge \beta_n$
and $x\in\{2,\ldots,n\}$, which is a precise version of an unproven 
claim in the cited reference.
\end{Remark}
\begin{Remark}                                              \label{Rem:beta_n_necessary}
The condition  $\beta\ge \beta_n$ in~\eqref{Eq:K_simple}  can not be omitted:
        For $n\in\N$, let $A_n\coloneqq\{\beta\in\mathopen]0,1\mathclose[ \colon 
        \text{If } \Kappa\text{ is as in Theorem~\ref{Thm:lower}, then }
        \Kappa(x)= \mathopen]g^{}_n(x),1\mathclose]\text{ for }x\in\{2,\ldots,n\}\}$. 
        Then $\mathopen[\beta_n,1\mathclose[ \subseteq A_n$, by  Theorem~\ref{Thm:lower}.
        Numerically, we found for example  also $\beta_n-0.001  \in A_n$ for $2\le n\le 123$,
        but  $\K(2)\supsetneq \mathopen]g^{}_{n}(2),1\mathclose]$ 
        for $\beta = \beta_n-0.001$ and $124\le n\le 3000$.
\end{Remark}
\begin{Remark}  \label{Rem:K_not_admissible_as_interval}         
The $\beta$-confidence upray $\K$ for~\eqref{stmdl} from Theorem~\ref{Thm:lower} 
considered merely as a $\beta$-confidence interval 
shares with $\K^{}_{\mathrm{CP,}n}$ as a $\beta$-confidence interval for~\eqref{binom2}
the defect of not being admissible in the set of \emph{all} 
$\beta$-confidence intervals, since with 
$c\coloneqq \left(\inf\K(n)\right)\vee \left(1-(1-\beta)^{1/n}\right)$
and  
\[
	\K^\ast(x) &\coloneqq&\left\{\begin{array}{ll}
                             [0,c] \,\subsetneq\, \Kappa(0) & \text{if } x=0,\\
			     \K(x) & \text{if } x\in\{1,\dots,n\},
                            \end{array}
	\right.
\]
we have $\bcp(\K^\ast\ni\pqu) = \bcp(\K\ni\pqu) \ge \beta$ if $\pqu\leq c$,
and, if $\pqu>c,$
$\bcp(\K^\ast\ni\pqu)$ $=$ $\bcp(\{1,\dots,n\})$ $=$ $1-\prod\nolimits_{j=1}^n (1-p_j)$ 
	$\geq$ $1-(1-\pqu)^n > 1-(1-c)^n\geq \beta.$ 

\end{Remark}
\begin{Remark}                                              \label{Rem:Upper_as_lower}
Since $\Kappa$ is a $\beta$-confidence region for~\eqref{stmdl} iff
$\{0,\ldots,n\}\ni x\mapsto 1-\Kappa(n-x)$ is one, Theorem~\ref{Thm:lower} and
Remarks~\ref{Rem:nested}--\ref{Rem:K_not_admissible_as_interval} yield obvious
analogs for \defn{downrays}, that is confidence regions with each value being
$\mathopen[0,b\mathclose[$ or $[0,b]$ for some $b\in[0,1]$: A downray 
$\Lambda\colon\{0,\ldots,n\}\rightarrow 2^{[0,1]}$  is \defn{isotone} if 
$\Lambda(x)\subseteq \Lambda(y)$ holds for $x<y$. The Clopper-Pearson downrays
$\Lambda^{}_{\mathrm{CP,}n} \coloneqq \Lambda^{}_{\mathrm{CP,}n,\beta}$ 
defined by $\Lambda^{}_{\mathrm{CP,}n,\beta}(x)\coloneqq 1- \Kappa^{}_{\mathrm{CP,}n,\beta}(n-x)$
are isotone, and Theorem~\ref{Thm:lower} remains valid if we replace  
$\Kappa^{}_{\mathrm{CP,}m}$ by $\Lambda^{}_{\mathrm{CP,}m}$, upray by downray, and~\eqref{Eq:K_simple} 
by
\la                              \label{Eq:K_simple_downray}
 \qquad
 \K(x)&=&\left\{\begin{array}{ll}
               \left[0,1-g^{}_{n}(n-x)\right[ & \text{if } x\in\{0,\dots,n-2\}
                        \text{ and }  \beta \geq \beta_n,     \\
                \left[0, 1-\frac{1-\beta}{n}\right[                   & \text{if } x=n-1,\\
	        \left[0,1\right]    & \text{if } x=n.
 
          \end{array}\right.
\al
\end{Remark}
\begin{Remark}                                           \label{Rem:Wrong_claims}                    
Papers erroneously claiming the Clopper-Pearson uprays or downrays
to be  $\beta$-confidence regions for~\eqref{stmdl} include 
Kappauf and Bohrer~\cite[p.~652, lines 3--5]{Kappauf}, 
Byers et al.~\cite[p.~249, first column, lines 15--18]{Byers}, 
and Cheng et al.~\cite[p.~7, lines  10--8 from the bottom]{Cheng}.
The analogous claim of Ollero and Ramos~\cite[p.~247, lines 9--12]{Ollero}  
for a certain subfamily of $(\bcp\colon p\in[0,1]^n),$
which includes the hypergeometric laws with sample size parameter $n,$
is refuted in Remark~\ref{Remark:Hyp_not_valid} below.
The common source of  error in these papers seems to be 
an unclear remark of Hoeffding~\cite[p.~720, first paragraph of section 5]{Hoef_1956}
related to the fact that, by \cite[Theorem 4]{Hoef_1956} or by David~\cite{David},
certain tests for $p\mapsto p$ in the binomial model $(\mathrm{B}_{n,p}\colon p\in[0,1])$
keep their level as tests for $p\mapsto\overline{p}$ in $(\bcp\colon p\in[0,1]^n)$.
Let us further note that \cite{Ollero} should have cited 
Vatutin and Mikhailov~\cite{Vatutin} concerning 
the representability of hypergeometric laws as Bernoulli convolutions.
\end{Remark}
\begin{Remark}                                           \label{Rem:Hoeff_tests}
The core of the unclear remark in \cite{Hoef_1956} mentioned in Remark~\ref{Rem:Wrong_claims}
is ``that the usual (one-sided and two-sided) tests for the constant probability of `success'
in $n$ independent (Bernoulli) trials can be used as tests for the average probability of success when the probability of success varies from trial to trial.''
We specify and generalise this in the following way:
Let $n \in\N,$ $p_1 \leq p_2 \in [0,1],$ $\gamma_-,\gamma_+\in[0,1],$ 
$c_-\leq\lfloor np_1\rfloor -1,$ and $c_+\geq \lceil np_2 \rceil +1.$
Then the randomised test
\begin{align*}
 \psi \coloneqq \1_{\{0,\dots,c_{-}-1\}} + \gamma_{-}\1_{\{c_{-}\}} + \gamma_{+}\1_{\{c_{+}\}} +\1_{\{c_{+}+1,\dots,n\}} 
\end{align*}
for the hypothesis $[p_1,p_2]$ in the binomial model $\left(\B_{n,p}\colon p\in[0,1]\right)$
keeps its level as a randomised test for
$\left\{p\in[0,1]^n\colon\pqu\in[p_1,p_2]\right\}$ in the model $\left(\bcp\colon p\in[0,1]^n\right),$
because for every $p$ with $\pqu \in[p_1,p_2]$ it follows from \cite[Theorem 4]{Hoef_1956} that we have
  \begin{align*}
   \bcp \psi\,\ = &\,\ \gamma_- \bcp(\{0,\dots,c_-\}) + (1-\gamma_-)\bcp(\{0,\dots,c_--1\})\\
	       &\,\  +\gamma_+ \bcp(\{c_+,\dots,n\}) + (1-\gamma_+)\bcp(\{c_++1,\dots,n\})\\
	     \leq &\,\ \B_{n,\pqu}  \psi.
  \end{align*}
But this statement does not always apply to the one-sided tests based on the
Clopper-Pearson uprays:\\ 
Let $n=2$ and $\beta\in\mathopen]0,1\mathclose[.$ 
Let $r\in\mathopen[0,1\mathclose],$ $H\coloneqq\mathopen[0,r\mathclose],$ and $\psi\coloneqq \1_{\{\K_{\mathrm{CP},n}\cap H =\emptyset\}},$
so that we have $\sup\nolimits_{p\in H}\B_{n,p}\psi\leq 1-\beta.$
But, if for example $r=1-\sqrt{\beta},$ the test simplifies to $\psi=\1_{\{1,2\}},$ 
and for $p\coloneqq(r-\varepsilon,r+\varepsilon)$ for an $\varepsilon>0$ small enough,
we have $\overline{p}\in H$ and $\bcp\psi = 1 -\bcp(\{0\}) = 1-\beta+\varepsilon^2 >1-\beta.$
\end{Remark}                    

\begin{Remark}                          \label{Remark:Hyp_not_valid}
Clopper-Pearson uprays can be invalid for hypergeometric estimation 
problems: 
For $N\in\N_0,$ $n\in\{0,\dots,N\}$, and 
$p\in\left\{\tfrac{j}{N}\colon 
j\in\left\{0,\dots,N\right\}\right\},$ let $\mathrm{H}^{}_{n,p,N}$ denote the 
hypergeometric law of the number of red balls drawn in a simple random sample of size $n$ 
from an urn containing $Np$ red and $N(1-p)$ blue balls, so that we have 
$\mathrm{H}^{}_{n,p,N}(\{k\}) = \binom{Np}{k} \binom{N(1-p)}{n-k} / \binom{N}{n}$ for $k\in\N_0.$
For $\beta \in\mathopen]0,1\mathclose[$ and fixed $n$ and $N,$ in general, $\K^{}_{\mathrm{CP,}n}$ 
is not a $\beta$-confidence region for the estimation problem 
$\left(\left(\mathrm{H}^{}_{n,p,N}\colon p\in \left\{\tfrac{j}{N}\colon
j\in\left\{0,\dots,N\right\}\right\}\right), p\mapsto p\right),$
because if, for example, $n \geq 2$ and $\beta=\left(1-\tfrac 1N \right)^n,$ then for
$p=g^{}_{n}(1)$ we have $p = 1-\beta^{1/n} = \tfrac 1N$ and so 
$\mathrm{H}^{}_{n,p,N}\left(\K^{}_{\mathrm{CP,}n}\ni p\right) = \mathrm{H}^{}_{n,p,N}\left(\{0\}\right)
= \binom{N(1-p)}{n} / \binom{N}{n} = \prod\nolimits_{j=0}^{n-1}\tfrac{N(1-p)-j}{N-j} < (1-p)^n = \beta.$
\end{Remark}
%
In contrast to Remark~\ref{Rem:KCP_not_robust}, we have the following positive
result for the two-sided Clopper-Pearson $\beta$-confidence intervals $\Mu^{}_{\mathrm{CP,}n}$
for~\eqref{binom2}, as defined in~\eqref{Eq:Def_Mu_CP} below.
\begin{Thm}                                   \label{Thm:Two-sided-CP_valid}
Let $n\in\N$, $\beta\in\mathopen]0,1\mathclose[$, and 
\la                                                    \label{Eq:Def_Mu_CP}
 \qquad \Mu^{}_{\mathrm{CP,}n}(x)
  &\coloneqq&\Kappa^{}_{\mathrm{CP,}n,\frac{1+\beta}{2}}(x) \cap \Lambda^{}_{\mathrm{CP,}n,\frac{1+\beta}{2}}(x)
\quad\text{ for }x\in\{0,\ldots,n\}
\al
with $\Kappa^{}_{\mathrm{CP,}n,\frac{1+\beta}{2}}$ as in \eqref{clop} and 
$\Lambda^{}_{\mathrm{CP,}n,\frac{1+\beta}{2}}$ as in Remark~\ref{Rem:Upper_as_lower}.
If $\beta \ge2\beta_n - 1$ or $n=1$, hence in particular if $\beta\ge\frac12$, then 
$\Mu^{}_{\mathrm{CP,}n}$ is a  $\beta$-confidence interval for~\eqref{stmdl}.
\end{Thm}
\begin{Remark}           \label{Rem:Two-sided_CP_Vs_Agnew} 
The interval $\Mu^{}_{\mathrm{CP,}n}$ of Theorem~\ref{Thm:Two-sided-CP_valid} 
improves on the two-sided interval for~\eqref{stmdl} obtained by Agnew~\cite{Agnew} 
in the obvious way from his one-sided ones. 
\end{Remark}

\begin{Remark}            						 \label{Rem:Joint}  
In contrast to Remark~\ref{Rem:beta_n_necessary}, we do not know whether the condition
``$\beta \ge2\beta_n - 1$ or $n=1$'' in Theorem~\ref{Thm:Two-sided-CP_valid} might be omitted.
\end{Remark}

\begin{Remark} 
The robustness property of the two-sided Clopper-Pearson intervals 
given by Theorem~\ref{Thm:Two-sided-CP_valid} does not extend to every other two-sided
interval  for~\eqref{binom2}, for example if $n=2$ not to the Sterne~\cite{Sterne} type 
$\beta$-confidence interval $\K^{}_{\mathrm{S,}n}$ 
for~\eqref{binom2} of D\"umbgen~\cite[p.~5, $C_\alpha^\mathrm{St}$]{Dumbgen}:
For $\beta \in \mathopen]0,1\mathclose[$ and $n\in\N,$ $\K^{}_{\mathrm{S,}n}$ is given by 
\begin{eqnarray*}
 \K^{}_{\mathrm{S,}n}(x)  &\coloneqq& \K^{}_{\mathrm{S,}n,\beta}(x)\\
    &\coloneqq& \left\{p\in\mathopen[0,1\mathclose] \colon 
      \B_{n,p}\left( \left\{k\colon \B_{n,p}(\{k\})\leq \B_{n,p}(\{x\})\right\} \right)  
     \,\ > \,\  1-\beta \right\}.
\end{eqnarray*}
If, for example, $n=2$ and $\beta > \beta_2$ we have in particular
 $\K^{}_{\mathrm{S,}2}(0) = \left[0, 1- g^{}_{2}(2)\right[,$ 
 $\K^{}_{\mathrm{S,}2}(1) = \left]g^{}_{2}(1), 1-g^{}_{2}(1)\right[,$ and
 $\K^{}_{\mathrm{S,}2}(2) = \left]g^{}_{2}(2),1\right],$
and indeed $\K^{}_{\mathrm{S,}2}$ is not valid for \eqref{stmdl}, 
because for $p\in\mathopen[0,1\mathclose]^2$ 
with $\pqu = g^{}_{2}(1)$ and $p_1 \neq p_2$ we have
\begin{align*}
	   \bcp\left(\K^{}_{\mathrm{S,}2}\ni\pqu\right) 
  =  \bcp\left(\left\{0\right\}\right)
  =  \prod\limits_{j=1}^2 (1-p_j)
  <  \left(1-\pqu\right)^2
  =  \left(1-g^{}_{2}(1)\right)^2
  =  \beta.	
\end{align*}
For $n=2$ and $\beta > \beta_2$ we get a $\beta$-confidence interval for \eqref{binom2},
say $\tilde{\K},$ from Theorem \ref{Thm:main}
by setting $\K'_m\coloneqq \K^{}_{\mathrm{S,}m}$
for $m\in\{0,1,2\},$ namely 
\begin{align*}
 \tilde\K(0) = \left[0,1-(1-\beta)^{1/2}\right[ \ ,\ 
 \tilde\K(1) = \left]\tfrac{1-\beta}{2},\tfrac{1+\beta}{2}\right[ \ ,\   
 \tilde\K(2) = \left](1-\beta)^{1/2},1\right].
\end{align*}
One computes that $\tilde\K(x) \subsetneq \Mu^{}_{\mathrm{CP,}2}(x)$ 
for $x\in\{0,1,2\},$ with $\Mu^{}_{\mathrm{CP,}2}$ as defined in 
Theorem \ref{Thm:Two-sided-CP_valid}.
We do not know whether these inclusions are true for every $n$ and usual $\beta,$ 
but in fact we do not even know whether 
$\K^{}_{\mathrm{S,}n} (x) \subseteq \Mu^{}_{\mathrm{CP,}n} (x)$ holds universally.
\end{Remark}
%
\section{Proofs of the theorems}              \label{Sec:2}
\begin{proof}[Proof of Theorem~{\ref{Thm:main}}]
We obviously have  $\K(x)\subseteq[0,1]$ and, by considering $l=0$ and $m=n$,  
$\Kappa(x)\supseteq\Kappa'_n(x)$ for every $x$. 
If $\phi\colon\{0,\ldots,n\}\rightarrow\R$ is any function and $\pi\in[0,1]$,
then, by Hoeffding's~(1956, Corollary 2.1) generalization of 
Tchebichef~\cite[second Th\'eor\`eme]{Chebyshev},
the minimum of the expectation $\mathrm{BC}_{p} \phi $ as a function of $p\in[0,1]^n$ subject to 
$\overline{p}=\pi$ is attained at some point $p$ whose coordinates take on 
at most three values and  with at most one of these distinct from $0$ and $1$. 
Given $p\in[0,1]^n$, the preceding sentence applied to $\pi\coloneqq\pqu$
and to $\phi$ being the indicator of $\{\K \ni \pi\}$ yields the existence of 
$r,s\in\{0,\ldots,n\}$ with $r+s\le n$ and of an $a\in[0,1]$
with $r+sa=n\pi$ and 
\[
 \mathrm{BC}_{p}\left(\K \ni \overline{p} \right)
 &\ge& \left(\delta_r\ast\mathrm{B}_{s,a}\right)
      \left(\{ x\in\{r,\ldots,r+s\} \colon \K(x) \ni \pi\}  \right) \\
 &\ge& \left(\delta_r\ast\mathrm{B}_{s,a}\right)
      \left(\{ x\in\{r,\ldots,r+s\} \colon \tfrac{s}{n}\K_s'(x-r)+\tfrac{r}{n} \ni \pi\}  \right) \\
 &=& \mathrm{B}_{s,a}\left(\K_s' \ni a \right) \\
 &\ge& \beta
\]  
by bounding in the second step the union defining $\K(x)$ by the set with the index $(l,m)=(r,s)$. 
\end{proof}
For proving Theorem~\ref{Thm:lower}, we use 
Lemma~\ref{Lem:Agnew_bounds_simplified}
prepared by Lemma~\ref{Lem:Binomial_inequality}. Let $F^{}_{n,p}$ and  $f^{}_{n,p}$
denote the  distribution and density
functions of the binomial law $\mathrm{B}_{n,p}$.

\begin{Lem}                                           \label{Lem:Binomial_inequality} 
Let $n\in\N$. Then
\la     \label{Eq:Binomial_inequality} 
   F^{}_{n,\frac{x}{n}}(x) &<& F^{}_{n,\frac{1}{n}}(1) \quad\text{ for }x\in\{2,\ldots,n-1\}.
\al
\end{Lem}
\begin{proof}
If $x\in\N$ with
$x\le \frac{n-1}2$, then for $p\in\mathopen]\frac{x}{n},\frac{x+1}{n}\mathclose[$,
we have $y\coloneqq x+1-np>0$, hence
  \begin{align*}
        \frac{f^{}_{n-1,p}\left(x\right)}{f^{}_{n,\frac{x+1}{n}}\left(x+1\right)}
  & =     \frac{f^{}_{n-1,p}\left(x\right)}{f^{}_{n-1,\frac{x+1}{n}}\left(x\right)}\\
  & =     \frac{\left(1+\frac{y}{n-x-1}\right)^{n-x-1}}{\left(1+\frac{y}{np}\right)^x}
   >   \frac{\left(1+\frac{y}{n-x-1}\right)^{n-x-1}}{\left(1+\frac{y}{x}\right)^x}
   \geq\,\ 1,
 \end{align*}
using the isotonicity of $\mathopen]0,\infty\mathclose[ \ni t \mapsto \left(1+\frac{y}t \right)^t$
in the last step,  and hence we get  
\[
  F^{}_{n,\frac{x}{n}}(x) -  F^{}_{n,\frac{x+1}{n}}(x+1) &=& 
  n\int\limits_{\frac{x}{n}}^{\frac{x+1}{n}}f^{}_{n-1,p}\left(x\right)\mathrm dp 
    \,-\, f^{}_{n,\frac{x+1}{n}}\left(x+1\right) \,\ >\,\  0;
\]
consequently~\eqref{Eq:Binomial_inequality} holds under the restriction $x\le \frac{n+1}2$.
If now $x\in\N$ with $\frac{n+1}2\le x\le n-1$, then $1\le k \coloneqq n-x <\frac{n}2$, 
and hence an inequality attributed to Simmons by Jogdeo and Samuels~\cite[Corollary 4.2]{Jogdeo} yields
$F^{}_{n,\frac{k}{n}}(k-1)  >  1-  F^{}_{n,\frac{k}{n}}(k)$,
so that
\[
  F^{}_{n,\frac{x}{n}}(x)&=& 1-  F^{}_{n,\frac{k}{n}}(k-1) \,\ <\,\ F^{}_{n,\frac{k}{n}}(k) 
  \,\ \le\,\ F^{}_{n,\frac{1}{n}}(1), 
\]
using in the last step~\eqref{Eq:Binomial_inequality} in a case already proved
in the previous sentence.
\end{proof}

\begin{Lem}                                               \label{Lem:Agnew_bounds_simplified}
Let $n\in\N$, $\beta \in [\beta_n,1[$,  and $x\in\{2,\ldots,n\}$. Then $g_n(x)\le\frac{x-1}n$.
\end{Lem}
\begin{proof} Using Lemma~\ref{Lem:Binomial_inequality}, we get 
$ F^{}_{n,\frac{x-1}{n}}(x-1) \le F^{}_{n,\frac{1}{n}}(1)=\beta_n \le \beta =F^{}_{n,g^{}_n(x)}(x-1)$,
and hence the claim.
\end{proof}

\begin{proof}[Proof of Theorem~\ref{Thm:lower}]
To simplify the defining representation of $\K$ in the present case, let us put
\la                                       \label{Eq:Def_g_proof_Thm_2}
 \qquad g(x) &\coloneqq& 
   \min\limits_{\genfrac{}{}{0pt}{}{l\in\{0,\dots,x-1\},}{m\in\{x-l,\dots,n-l\}}}
    \left(\tfrac mn g^{}_{m}(x-l) + \tfrac ln \right)
  \qquad\text{ for }x\in\{1,\ldots,n\}.
\al
For $x\in\{0,\ldots,n\}$, we have, using \eqref{clop}, 
\[
 \K(x)&\supseteq&\tfrac{n-x}n\K_{\mathrm{CP},n-x}(x-x)+\tfrac{x}n \,\ = \,\ \left[\tfrac{x}n,1\right],
\] 
hence in particular $\K(0)=[0,1]$. For $x\in\{1,\ldots,n\}$, we have,
with $(l,m)$ denoting some pair where the minimum in~\eqref{Eq:Def_g_proof_Thm_2}
is attained, 
\[
 \K(x)&\supseteq&\tfrac{m}n\K_{\mathrm{CP},m}(x-l) +\tfrac{l}n
  \,\ =\,\ \left]g(x),\tfrac{l+m}n\right] 
  \,\ \supseteq \,\  \left]g(x),\tfrac{x}n\right]
\]
and, using $g^{}_{x}(x)<1$ at the third step below,
\[
\K(x) \setminus \left]g(x),1\right] 
 &\subseteq& \bigcup_{m\in\{0,\ldots,n-x\}} 
  \left( \tfrac{m}n \K_{\mathrm{CP},m}(x-x)+\tfrac{x}n \right)
 \,\ \subseteq \,\ \left[\tfrac{x}n,1  \right] \\
 &\subseteq& \left] \tfrac{x}n g^{}_{x}(x-0)+\tfrac0n    , 1\right]
  \,\ \subseteq\,\ \mathopen]g(x),1\mathclose].
\]
Combining the above yields
\la                      \label{Eq:Rep_Kappa_g}
 \K(x) &=&  \left\{\begin{array}{ll} [0,1] &\text{ if }x=0,\\
      \mathopen]g(x),1\mathclose] &\text{ if }x\in\{1,\ldots,n\},
  \end{array}\right.
\al
so in particular $\K$ is indeed an upray, and~\eqref{Eq:K_simple} 
holds in its trivial first case. 
Using~\eqref{Eq:gn(1)} and the isotonicity of $t\mapsto\left(\beta^t-1\right)/t$ due to 
the convexity of $t\mapsto \beta^t$ yields 
\[
 g(1) &=& \min_{m=1}^n \tfrac{m}{n}g^{}_{m}(1)
 \,\ = \,\ \tfrac1n  \min_{m=1}^n m\left(1-\beta^{ 1/m }\right) \,\ = \,\ \tfrac{1-\beta}n
\]
and hence~\eqref{Eq:K_simple} also in the second case. The last case 
is treated at the end of this proof. 

$\Kappa$ is strictly isotone, since, for   $x\in\{2,\dots,n\}$,
we get, using  $g^{}_{m}(x-1) < g^{}_{m}(x)$ for 
$2 \leq x \leq m \leq n$ due to~\eqref{Eq:g_x_beta_monotone},  
\[ 
  g(x)& = & \min\limits_{m\in \{x,\dots,n\}} \tfrac mn g^{}_{m}(x) \nonumber\\
      & &  \wedge \min\limits_{\genfrac{}{}{0pt}{}{l\in \{1,\dots,x-1\},}{m\in \{x-(l-1)-1,\dots,n-(l-1)-1\}}} 
	 \left(	\tfrac mn g^{}_{m}(x-1-(l-1)) +\tfrac {l-1}{n} +\tfrac 1n\right) \nonumber\\
        & >  & \min\limits_{m\in \{x-1,\dots,n\}} \tfrac mn g^{}_{m}(x-1)
                \wedge \!\!\! \min\limits_{\genfrac{}{}{0pt}{}{l\in \{0,\dots,x-1-1\},}{m\in \{x-1-l,\dots,n-1-l\}}} \! \!\!
       \left(\tfrac mn g^{}_{m}(x-1-l) +\tfrac {l}{n} \right) \nonumber\\
        & \geq  & g(x-1).
\]

By considering $p=(1-\beta,0,\ldots,0)\in[0,1]^n$ at the first step below, and using 
$\Kappa(1)=\mathopen]\frac{1-\beta}n,1\mathclose]\not\ni\tfrac{1-\beta}{n}$ and the isotonicity of $\K$
at the second, we get
\[
 \inf_{p\in[0,1]^n}\mathrm{BC}_{p}(\Kappa\ni\overline{p})
 &\le& \mathrm{B}_{1-\beta}\left(\Kappa\ni\tfrac{1-\beta}{n}\right)
 \,\ = \,\ \mathrm{B}_{1-\beta}\left(\{0\}\right) \,\ =\,\ \beta
\]
and hence, by Theorem \ref{Thm:main}, $\inf_{p\in[0,1]^n}\mathrm{BC}_{p}(\Kappa\ni\overline{p})=\beta$.

To prove the optimality of $\Kappa$, let us assume that 
$\tilde\K\colon\{0,\ldots,n\}\rightarrow 2^{[0,1]}$ is another isotone upray and 
that we have an $x'\in\{0,\ldots,n\}$ with 
\la                              \label{Eq:Assumption_K*}
 \tilde\K(x')  \,\ \subsetneq \,\  \K(x').
\al
We have to show that  $\inf_{p\in[0,1]^n} \bcp(\tilde\K \ni \pqu) < \beta$.
If $x'=0$, then $\Kappa(x')=[0,1]$ and, since  $\tilde\K(0)$ is an upray in $[0,1]$,
\eqref{Eq:Assumption_K*} yields
$0\notin \tilde\K(0)$, and hence  
\[
 \inf_{p\in[0,1]^n} \bcp \left(\tilde\K \ni \pqu\right) 
&\leq& \delta_{0}\left(\tilde\K \ni 0 \right)\,\ =\,\ 0 \,\ <\,\ \beta.
\]
If $x'\in \{1,\dots,n\}$, then, using~\eqref{Eq:Rep_Kappa_g} and~\eqref{Eq:Def_g_proof_Thm_2}, we get 
$\Kappa(x')=\mathopen]\frac{m}{n}g^{}_{m}(x'-l)+\frac{l}{n},1\mathclose]$ for some 
$l\in\{0,\dots,x'-1\}$ and $m\in\{x'-l,\dots,n-l\}$, and since  $g^{}_{m}(x'-l)<1$, we find an 
$a\in \mathopen]g^{}_{m}(x'-l),1\mathclose]$ with $\frac{m}{n}a+\frac{l}{n}
\notin\tilde\K(x')$, hence $\frac{m}{n}a+\frac{l}{n}
\notin\tilde\K(y)$ for $y\in\{x',\ldots, n\}$ by the isotonicity of $\tilde\K$,
and hence
\[
 \inf\limits_{p\in[0,1]^n} \bcp(\tilde\K \ni \pqu)
 &\le&  \B_{m,a}\left(\left\{x\in\{0,\dots,n\}\colon\tilde\K(x+l)\ni\tfrac{l+ma}{n}\right\}\right)\\
 &\le&  \B_{m,a}(\{0,\dots,x'-l-1\}) \\
 & < &  \B_{m,g^{}_{m}(x'-l)}(\{0,\dots,x'-l-1\})\\
 & = &  \beta.
\]

To prove the admissibility of $\K,$ assume that there was a $\beta$-confidence upray $\K^\ast$ 
for~\eqref{stmdl} with $\K^\ast(x)\subseteq\K(x)$ for each $x\in\{0,\dots,n\}$ and
$\K^\ast(x')\subsetneq\K(x')$ for some $x'.$ Then, since 
$\K$ is strictly isotone,  
\[
 \K^{\ast\ast}(x) &\coloneqq&\left\{\begin{array}{ll} \Kappa(x)&\text{ if }x\neq x'\\
        \K^\ast(x')\cup\Kappa(x'+1)&\text{ if }x= x' < n,\\
        \K^\ast(x') &\text{ if }x= x' = n
 \end{array}\right\}
 \,\ \supseteq \,\ \Kappa^{\ast}(x)  
\]
would define an isotone $\beta$-confidence upray for~\eqref{stmdl} with
$\Kappa^{\ast\ast}(x')\subsetneq\Kappa(x')$, contradicting  the optimality of $\Kappa$.

To prove finally the last case of~\eqref{Eq:K_simple}, let
$n\ge 2$ and $\beta \ge \beta_n$, and 
let now $\tilde\K\colon\{0,\ldots,n\}\rightarrow 2^{[0,1]}$ 
be defined by the right hand side of~\eqref{Eq:K_simple}. 
If $p\in[0,1]^n$ with 
$\pqu\in[0,\frac{1-\beta}{n}],$ then
 \[
  \bcp(\tilde\K\ni\pqu)&\geq& \bcp(\{0\}) \,\ =\,\ \prod\limits_{j=1}^n (1-p_j) 
   \,\ \geq\,\  1-\sum\limits_{j=1}^n p_j \,\ =\,\  1-n\pqu \,\ \geq \,\  \beta.
 \]
If $p\in[0,1]^n$ with 
$\pqu \in \mathopen]\frac{1-\beta}{n},1\mathclose]$, then with 
$g^{}_n(n+1)\coloneqq 1$ either there is a $c\in\{2,\ldots,n\}$
with $\pqu \in \mathopen]g^{}_n(c),  g^{}_n(c+1)  \mathclose]$,
or $\pqu \in \mathopen]\frac{1-\beta}{n},g^{}_{n}(2)\mathclose]$ and we put
$c\coloneqq1$; in either case then  
$n \pqu \leq n g^{}_{n}(c+1) \leq 1$ by Lemma~\ref{Lem:Agnew_bounds_simplified},
and hence an application of Hoeffding~\cite[Theorem~4, (26)]{Hoef_1956} 
at the second step below yields
\[
  \bcp\left(\tilde\K\ni\pqu\right) &=& \bcp\left(\{0,\ldots,c\}\right) 
  \,\ \geq \,\  F^{}_{n,\pqu} \left(c\right) \,\ \geq \,\ 
  F^{}_{n,g^{}_{n}(c+1)} \left(c\right)
  \,\   \ge  \,\   \beta.
\]
Hence $\tilde\K$ is a $\beta$-confidence upray for~\eqref{stmdl} and satisfies 
$\tilde\K(x)\subseteq\K(x)$ for each $x$, and so the admissibility of $\K$ yields
$\tilde\K=\K$, and hence~\eqref{Eq:K_simple}.
\end{proof} 

\begin{proof}[Theorem~\ref{Thm:Two-sided-CP_valid}]
Let $\gamma\coloneqq\frac{1+\beta}2$,  
let $\Kappa_\gamma$ be the $\gamma$-confidence upray from Theorem~\ref{Thm:lower},
and let $\Lambda_\gamma$ be the analogous  $\gamma$-confidence downray from 
Remark~\ref{Rem:Upper_as_lower}. Then, by subadditivity, 
$\Mu_\beta(x)\coloneqq \Kappa_\gamma(x)\cap \Lambda_\gamma(x)$ for $x\in\{0,\ldots,n\}$
defines a $\beta$-confidence interval for~\eqref{stmdl}. 
If $n=1$, then $\Kappa^{}_{\mathrm{CP},n} = \Mu_\beta$, hence the claim.
So let $\beta\geq2\beta_n-1$, that is, $\gamma\ge \beta_n$.
Then~\eqref{Eq:K_simple} and \eqref{Eq:K_simple_downray}, with 
$\gamma$ in place of $\beta$, yield
$\Mu^{}_{\mathrm{CP},n}(x) = \Mu^{}_\beta (x)$ for  $x\notin\{1,n-1\}.$
So, if $\pqu \notin \left( \Mu^{}_{\mathrm{CP},n}(1) \setminus \Mu^{}_\beta (1)\right)\cup 
		    \left( \Mu^{}_{\mathrm{CP},n}(n-1) \setminus \Mu^{}_\beta (n-1)\right),$ we have
$\bcp\left(\Mu^{}_{\mathrm{CP},n}\ni\pqu\right) = \bcp\left(\Mu^{}_\beta\ni\pqu\right)\geq \beta.$
Otherwise $\pqu \in \left]\frac{1-\gamma}{n},g^{}_{n,\gamma}(1)\right]$ 
or        $\pqu\in\left[g^{}_{n,\gamma}(n-1),1-\frac{1-\gamma}{n}\right[.$
In the first case, 
$\pqu \in \left] \frac{1-\gamma}{n}, g^{}_{n,\gamma}(1) \right]
       =  \left] \frac{1-\gamma}{n}, 1-\gamma^{1/n} \right]
\subseteq \left[0,1-(1-\gamma)^{1/n}\right[ = \Mu^{}_{\mathrm{CP},n}(0)$
and \\from $\pqu \in \Mu^{}_{\mathrm{CP},n}(0)$ and $\pqu\leq 1-\gamma^{1/n}$
we get
\begin{align*}
 \bcp\left(\Mu^{}_{\mathrm{CP},n}\ni\pqu\right)  \geq  \bcp(\{0\}) 
						& =     \prod\limits_{j=1}^n (1-p_j) \\
						& \geq  1-n\overline{p} 	
						 \geq  
							     1-n\left(1-\gamma^{1/n}\right) 
						 \geq  \gamma  
						 >     \beta.
\end{align*}
In the second case, analogously, 
$\pqu \in \left[g^{}_{n,\gamma}(n-1),1-\frac{1-\gamma}{n}\right[
      =   \left[\gamma^{1/n},1-\frac{1-\gamma}{n}\right[
      \subseteq   \Mu^{}_{\mathrm{CP},n}(n)$
and from $\pqu \in \Mu^{}_{\mathrm{CP},n}(n)$ and 
$\pqu \geq \gamma^{1/n}$ we get
$\bcp\left(\Mu^{}_{\mathrm{CP},n}\ni\pqu\right)  \geq \bcp(\{n\}) 
						= \prod\nolimits_{j=1}^n p_j 
						\geq {\pqu}^{\, n}
						\geq \gamma  
						>  \beta.$
\end{proof}
\section*{Acknowledgement}
We thank Jona Schulz for help with the proof of Lemma~\ref{Lem:Binomial_inequality},
and the referee for suggesting to address nestedness.

\end{document}